\documentclass[11pt]{amsart}
\usepackage{amsmath, amsthm, amscd, amsfonts, amssymb, graphicx, color}
\usepackage[ colorlinks, plainpages]{hyperref}
\usepackage{tikz}
\usepackage{booktabs}
\usepackage{placeins}
\usepackage{pgfplots}
\usepackage{caption}
\usepackage{algorithm}
\newtheorem{question}{Question}
\usetikzlibrary{automata}
\newtheorem{theorem}{Theorem}[section]
\newtheorem{lemma}[theorem]{Lemma}
\newtheorem{proposition}[theorem]{Proposition}
\newtheorem{corollary}[theorem]{Corollary}
\theoremstyle{definition}
\newtheorem{definition}[theorem]{Definition}
\newtheorem{example}[theorem]{Example}

\theoremstyle{remark}
\newtheorem{remark}[theorem]{Remark}
\numberwithin{equation}{section}
\begin{document}
\title[]{Defect Antichains and Multigraded Symbolic Defect Series of Edge Ideals under Graph Blow-ups}
\author[]{Tabinda Rasheed$^1$, Wang Yao$^{1}$}
\maketitle
{\centering{$^{1}$ School of Mathematics and Statistics, Nanjing University of Information Science and Technology, Nanjing 210044, P.R. China.}\\
\footnotesize{Corresponding author. Email address:
\href{mailto:xx@xx}{tabindarasheed00@gmail.com}, wangyao@nuist.edu.cn}}
\begin{abstract}
In this paper, we study symbolic defect functions of edge ideals through finite antichains of exponent vectors. Let $G$ be a finite simple graph and let $I(G)$ be its edge ideal. For each symbolic degree $s$, we define the symbolic exponent region $\mathcal{P}_s(G)$, the ordinary exponent region $\mathcal{O}_s(G)$, and the symbolic defect antichain $\mathcal{D}_s(G)=\min\big(\mathcal{P}_s(G)\setminus \mathcal{O}_s(G)\big)$, where the minimum is taken with respect to the componentwise partial order. We prove that $\mathcal{D}_s(G)$ gives a finite obstruction set controlling the minimal monomial generators of the quotient $I(G)^{(s)}/I(G)^s$.
Our main result is a blow-up transfer formula. If $G^{\mathbf n}$ is the graph obtained from $G$ by replacing each vertex $v_i$ by an independent set of size $n_i$, then for every $s\geq 1$,
\[
\operatorname{sdefect}(I(G^{\mathbf n}),s)
=
\sum_{\mathbf a\in \mathcal D_s(G)}
\prod_{i=1}^{r}
\binom{a_i+n_i-1}{n_i-1}.
\]
We further refine this formula to a multigraded symbolic defect series, which records the full multidegree distribution of the minimal generators of $I(G^{\mathbf n})^{(s)}/I(G^{\mathbf n})^s$.
As applications, we classify the defect antichains of complete graphs in terms of integer partitions and derive explicit symbolic defect formulas for complete multipartite graphs, complete split graphs, and blow-ups of odd cycles. We also study symbolic defect antichains under graph joins and obtain polynomiality and rational generating-function consequences in the blow-up parameters. The results provide a unified antichain-based framework for symbolic defects of edge ideals and convert several previously case-by-case computations into consequences of a single transfer principle.
\end{abstract}


\section{Introduction}
\label{sec:introduction}

Symbolic powers and ordinary powers of ideals are central objects in commutative algebra and algebraic combinatorics. Let $S=K[x_1,\ldots,x_n]$ be a polynomial ring over a field $K$, and let $I\subseteq S$ be a homogeneous ideal. The comparison between the symbolic power \(I^{(s)}\) and the ordinary power \(I^s\) is closely related to containment problems, symbolic Rees algebras, asymptotic invariants, and the structure of minimal generators. Symbolic powers have classical roots in the work of Krull, Zariski, and Nagata, and they continue to play an important role in modern commutative algebra through containment theory and the study of symbolic blow-up algebras \cite{KRU, ZAR, NAG, EIS, DAO}. In the monomial setting, symbolic powers admit polyhedral and combinatorial descriptions, and symbolic Rees algebras are closely connected with vertex cover algebras and affine semigroup methods \cite{HER, GIT, MIL, VILL, COO}.

Edge ideals of finite simple graphs provide a natural framework in which the difference between symbolic and ordinary powers becomes highly combinatorial. If \(G\) is a finite simple graph and \(I(G)\) is its edge ideal, then symbolic powers of \(I(G)\) are governed by vertex-cover inequalities, whereas ordinary powers are governed by products of edges. Thus, the quotient \(I(G)^{(s)}/I(G)^s\) reflects the gap between vertex-cover constraints and edge-incidence decompositions. This connection has been developed through the study of Rees algebras, associated primes, cover ideals, critical graphs, and graph operations \cite{SIM, VIL, FRA, SUL, GU}. Related developments also connect symbolic powers of edge ideals with Waldschmidt constants, resurgence-type invariants, and containment problems \cite{BOCC, COO, MINH}.

The behavior of symbolic powers of edge ideals is also closely connected with homological and asymptotic invariants. The Castelnuovo-Mumford regularity of ordinary powers has been studied for many graph classes, including forests, cycles, and more general families \cite{BAN, BEY}. For symbolic powers, regularity and depth have been investigated for edge ideals, cover ideals, unicyclic graphs, Cameron-Walker graphs, and other special graph classes \cite{JAY, SEY, SEYE, DUNG, MINHC}. Recent work has further emphasized componentwise linearity, regularity comparisons, and symbolic powers of edge ideals in connection with Minh's conjecture \cite{MIN, FICA, FIC, FAK, FAKH}. These results show that powers of edge ideals are naturally studied through a combination of algebraic, asymptotic, and graph-theoretic methods.
A numerical invariant measuring the failure of equality between symbolic and ordinary powers is the symbolic defect
$\operatorname{sdefect}(I,s)=\mu\bigl(I^{(s)}/I^s\bigr)$, where \(\mu\) denotes the minimal number of generators. Thus, while containment problems ask whether symbolic powers lie inside ordinary powers or their variants, the symbolic defect records the minimal new generators appearing in \(I^{(s)}\) but not in \(I^s\). 

Symbolic defect functions have recently attracted attention as invariants with both computational and asymptotic significance. Drabkin and Guerrieri studied symbolic defect functions for ideals with Noetherian symbolic Rees algebra and proved eventual quasi-polynomial behavior in several settings \cite{DRA}. Oltsik developed symbolic defect theory for monomial ideals using symbolic polyhedra and asymptotic methods \cite{OLT}. For edge ideals, exact symbolic defects have been computed for unicyclic graphs and related odd-cycle families \cite{GU, MAN}. These works show that symbolic defect is not merely a numerical byproduct but an invariant encoding how symbolic powers fail to coincide with ordinary powers.
To the best of our knowledge, symbolic defect functions of edge ideals have not previously been studied through finite antichains of exponent obstructions, nor has a general blow-up transfer formula for such antichains been established.

The present paper takes a complementary point of view for edge ideals. Instead of studying only the number \(\operatorname{sdefect}(I(G),s)\), we study the exponent patterns responsible for the quotient \(I(G)^{(s)}/I(G)^s\). For a finite simple graph \(G\), we introduce two exponent regions. The symbolic region \(\mathcal P_s(G)\) consists of exponent vectors satisfying the symbolic-cover inequalities determined by the minimal vertex covers of \(G\). The ordinary region \(\mathcal O_s(G)\) consists of exponent vectors that dominate an \(s\)-fold sum of edge-incidence vectors. Their difference
\[
\mathcal E_s(G)=\mathcal P_s(G)\setminus\mathcal O_s(G)
\]
records exponent vectors that occur symbolically but not ordinarily. We define the symbolic defect antichain
\[
\mathcal D_s(G)=\min_{\leq}\mathcal E_s(G),
\]
where the minimum is taken with respect to the componentwise partial order. Thus \(\mathcal D_s(G)\) is the finite set of minimal exponent obstructions responsible for the minimal monomial generators of \(I(G)^{(s)}/I(G)^s\).

The main theme of this paper is that the antichain \(\mathcal D_s(G)\) is the natural object for studying symbolic defects under graph blow-ups. Let \(G^{\mathbf n}\) be the graph obtained from \(G\) by replacing each vertex \(v_i\) by an independent set of size \(n_i\), and by replacing each edge of \(G\) with all edges between the corresponding independent sets. This operation includes several important graph families as special cases. Complete multipartite graphs are blow-ups of complete graphs, complete split graphs arise as special complete-graph blow-ups, and blow-ups of odd cycles are obtained by taking the base graph to be an odd cycle. Nevertheless, the results of this paper are not restricted to any one of these families: the base graph \(G\) is arbitrary. This viewpoint is closely related to the use of parallelizations and graph blow-ups in the study of symbolic powers of edge ideals \cite{SUL, COOP}.
Our first main result is a blow-up transfer formula for symbolic defects. If \(J=I(G^{\mathbf n})\), then for every \(s\geq 1\),
\[
\operatorname{sdefect}(J,s)
=
\sum_{\mathbf a\in\mathcal D_s(G)}
\prod_i
\binom{a_i+n_i-1}{n_i-1}.
\]
Thus the symbolic defect of every blow-up \(G^{\mathbf n}\) is determined by the defect antichain of the base graph \(G\). The binomial factor counts the number of ways an obstruction vector \(\mathbf a\) can be lifted to monomials in the blown-up variables. This separates the symbolic defect problem into two parts: the base obstruction data \(\mathcal D_s(G)\), which depends only on \(G\) and \(s\), and the blow-up weights, which depend only on the sizes of the blown-up parts.
We further refine this numerical transfer formula to a multigraded defect-generator series. If \(\mathcal M_{J,s}(\mathbf z)\) denotes the generating series of the minimal monomial generators of \(J^{(s)}/J^s\), then
\[
\mathcal M_{J,s}(\mathbf z)
=
\sum_{\mathbf a\in\mathcal D_s(G)}
\prod_i h_{a_i}(z_{i1},\ldots,z_{in_i}),
\]
where \(h_d\) is the complete homogeneous symmetric polynomial of degree \(d\). Consequently, the numerical symbolic defect is obtained by setting all auxiliary variables equal to \(1\). This shows that the blow-up formula is not only a stars-and-bars count; it is the numerical specialization of a multigraded antichain formula that describes the full multigraded distribution of the minimal-defect generators.

A second main contribution is a complete classification of the defect antichains of complete graphs. We prove that \(\mathcal D_s(K_t)\) is governed by integer partition types. More precisely, its elements are exactly the orbits, under permutation of coordinates, of partitions of total degree \(s+p\) whose largest part is \(p\), where \(1\leq p\leq s-1\), and whose largest part occurs at least twice. This gives an all-symbolic-degree description of \(\mathcal D_s(K_t)\). Since $K_{n_1,\ldots,n_t}=K_t^{(n_1,\ldots,n_t)}$, the classification immediately yields all-degree formulas for symbolic defects and multigraded defect series of complete multipartite graphs. This connects the present work with recent studies of symbolic powers of complete graphs and parallelizations of finite simple graphs \cite{COOP}.
The paper also records further consequences of the antichain viewpoint. We describe symbolic and ordinary exponent regions for graph joins and obtain a finite criterion for defect antichains of joins. This shows that the antichain language is compatible with graph operations beyond blow-ups. We also derive polynomiality and rational generating-function consequences in the blow-up parameters. For fixed \(G\) and \(s\), the function
\[
(n_1,\ldots,n_r)\longmapsto
\operatorname{sdefect}(I(G^{\mathbf n}),s)
\]
is polynomial, and the corresponding multivariable generating function is rational. A conditional quasi-polynomial transfer principle is also obtained: whenever weighted counts of the base antichains are eventually quasi-polynomial in \(s\), the symbolic defect functions of all fixed blow-ups inherit this behavior. This connects the present framework with the broader asymptotic study of symbolic defect functions \cite{DRA, MAN, OLT}.

The main contributions of the paper may be summarized as follows:
\begin{enumerate}
\item We introduce symbolic defect antichains
\[
\mathcal D_s(G)=\min_{\leq}
\bigl(\mathcal P_s(G)\setminus\mathcal O_s(G)\bigr)
\]
as finite obstruction sets controlling the quotient \(I(G)^{(s)}/I(G)^s\).

\item We prove a transfer formula for symbolic defects of arbitrary graph blow-ups \(G^{\mathbf n}\).

\item We refine the numerical symbolic defect formula to a multigraded defect-generator series.

\item We classify \(\mathcal D_s(K_t)\) for complete graphs in all symbolic degrees using partition types.

\item We derive applications to complete multipartite graphs, complete split graphs, and blow-ups of odd cycles, and obtain polynomiality and rational generating-function consequences in the blow-up parameters.

\item We show that the antichain framework is compatible with graph joins, providing a finite criterion for symbolic defect antichains under this operation.
\end{enumerate}

The paper is organized as follows. Section~\ref{sec:defect-regions} introduces the symbolic and ordinary exponent regions, defines the symbolic defect antichain, and proves the blow-up transfer formula. Section~\ref{sec:multigraded-defect-series} develops the multigraded defect-generator series. Section~\ref{sec:partition-complete-graphs} classifies the defect antichains of complete graphs. Section~\ref{sec:applications-blowup} gives applications of the transfer formula to complete multipartite graphs, complete split graphs, and blow-ups of odd cycles. Section~\ref{sec:joins} discusses defect regions under graph joins. Section~\ref{sec:asymptotic} proves polynomiality and rational generating-function consequences. The final section contains examples and computations illustrating the transfer principle.

\section{Defect regions and blow-up transfer}\label{sec:defect-regions}

\begin{lemma}\label{lem:min-covers-blowup}
Let \(G\) be a finite simple graph with vertex set $V(G)=\{v_1,\ldots,v_r\}$, and let \(G^{\mathbf n}\) be the \(\mathbf n\)-blow-up of \(G\), where \(v_i\) is replaced by the independent set $X_i=\{x_{i1},\ldots,x_{in_i}\}$.
Then the minimal vertex covers of \(G^{\mathbf n}\) are precisely the sets
\begin{equation}\label{eqsan}
C^{\mathbf n}=\bigcup_{v_i\in C}X_i,
\end{equation}
where \(C\) runs through the minimal vertex covers of \(G\).
\end{lemma}

\begin{proof}
Let \(W\) be a vertex cover of \(G^{\mathbf n}\). If \(\{v_i,v_j\}\in E(G)\), then \(W\) must contain all vertices of \(X_i\) or all vertices of \(X_j\); otherwise, choosing one vertex outside \(W\) from each of \(X_i\) and \(X_j\) would give an uncovered edge. Now assume that \(W\) is minimal. The previous observation implies that each
part \(X_i\) is either entirely contained in \(W\) or disjoint from \(W\). Indeed, if \(W\) contains some but not all vertices of \(X_i\), then all neighboring parts of \(X_i\) must be contained in \(W\), so the vertices of \(W\cap X_i\) are redundant, contradicting minimality.
Thus
\[
W=\bigcup_{v_i\in C}X_i
\]
for some subset \(C\subseteq V(G)\). The first paragraph shows that \(C\) is a vertex cover of \(G\), and minimality of \(W\) is equivalent to minimality of \(C\). Hence \(C\) is a minimal vertex cover of \(G\).
Conversely, if \(C\) is a minimal vertex cover of \(G\), then Eq. \eqref{eqsan} covers every edge of \(G^{\mathbf n}\). Minimality follows because for every \(v_i\in C\), minimality of \(C\) gives an edge \(\{v_i,v_j\}\in E(G)\) with \(v_j\notin C\); hence every vertex of \(X_i\) is needed to cover some edge to \(X_j\). Therefore \(C^{\mathbf n}\) is a minimal vertex cover of \(G^{\mathbf n}\).
\end{proof}

\begin{lemma}\label{lem:symbolic-membership-part-degree}
Let \(G\) be a finite simple graph, let \(G^{\mathbf n}\) be its
\(\mathbf n\)-blow-up, and set $S_{\mathbf n}=K[x_{ij}:1\leq i\leq r,\ 1\leq j\leq n_i]$. For a monomial
\[
u=\prod_{i=1}^{r}\prod_{j=1}^{n_i}x_{ij}^{b_{ij}}\in S_{\mathbf n},
\]
define its part-degree vector by $\pi(u)=(a_1,\ldots,a_r)$, $a_i=\sum_{j=1}^{n_i}b_{ij}$.
For \(s\geq 1\), define
\[
\mathcal P_s(G)
=
\left\{
\mathbf a\in\mathbb N^r:
\sum_{v_i\in C}a_i\geq s
\text{ for every minimal vertex cover }C\text{ of }G
\right\}.
\]
Then
\[
u\in I(G^{\mathbf n})^{(s)}
\quad\Longleftrightarrow\quad
\pi(u)\in\mathcal P_s(G).
\]
\end{lemma}

\begin{proof}
Since \(I(G^{\mathbf n})\) is squarefree, its symbolic power is
\[
I(G^{\mathbf n})^{(s)}
=
\bigcap_{W\in\mathcal C(G^{\mathbf n})}P_W^s,
\]
where \(W\) runs through the minimal vertex covers of \(G^{\mathbf n}\).
By Lemma~\ref{lem:min-covers-blowup}, these covers are exactly in \eqref{eqsan}, where $C\in\mathcal C(G)$.
For a fixed minimal cover \(C\) of \(G\), the monomial \(u\) belongs to \(P_{C^{\mathbf n}}^s\) if and only if its total degree in the variables of \(C^{\mathbf n}\) is at least \(s\). This total degree is
\[
\sum_{v_i\in C}\sum_{j=1}^{n_i}b_{ij}
=
\sum_{v_i\in C}a_i.
\]
Therefore \(u\in I(G^{\mathbf n})^{(s)}\) if and only if $\sum_{v_i\in C}a_i\geq s$ for every minimal vertex cover \(C\) of \(G\), which is exactly \(\pi(u)\in\mathcal P_s(G)\).
\end{proof}

\begin{lemma}\label{lem:ordinary-membership-part-degree}
Let \(G\) be a finite simple graph and let \(G^{\mathbf n}\) be its
\(\mathbf n\)-blow-up. Let \(A_G\) denote the vertex-edge incidence matrix of
\(G\). For \(s\geq 1\), define
\[
\mathcal O_s(G)
=
\left\{
\mathbf a\in\mathbb N^r:
\mathbf a\geq A_G\boldsymbol\lambda
\text{ for some } \boldsymbol\lambda\in\mathbb N^{E(G)}
\text{ with } \sum_{e\in E(G)}\lambda_e=s
\right\}.
\]
Then, for every monomial \(u\in S_{\mathbf n}\),
\[
u\in I(G^{\mathbf n})^s
\quad\Longleftrightarrow\quad
\pi(u)\in\mathcal O_s(G).
\]
\end{lemma}

\begin{proof}
Suppose first that \(u\in I(G^{\mathbf n})^s\). Then \(u\) is divisible by a
product of \(s\) edge monomials of \(G^{\mathbf n}\). Projecting these \(s\)
edges to \(G\), let \(\lambda_e\) be the number of times the edge
\(e\in E(G)\) occurs. Then
\begin{equation}\label{eqsan1}
\sum_{e\in E(G)}\lambda_e=s.
\end{equation}
The vector \(A_G\boldsymbol\lambda\) records how many selected edge factors use each part \(X_i\). Since the product divides \(u\), these numbers are bounded by the part-degrees of \(u\). Hence $\pi(u)\geq A_G\boldsymbol\lambda$,
so \(\pi(u)\in\mathcal O_s(G)\).
Conversely, suppose \(\pi(u)=\mathbf a\in\mathcal O_s(G)\). Then there exists \(\boldsymbol\lambda\in\mathbb N^{E(G)}\) with Eq. \eqref{eqsan1} and $\mathbf a\geq A_G\boldsymbol\lambda$.
For each \(i\), the \(i\)-th component of \(A_G\boldsymbol\lambda\) is the
number of edge occurrences incident to \(v_i\). Since this number is at most
the total degree of \(u\) in \(X_i\), we may choose the required number of
variables from the multiset of variables appearing in \(u\) inside \(X_i\).
For each occurrence of an edge \(\{v_i,v_j\}\), choose one selected variable
from \(X_i\) and one from \(X_j\). This gives \(s\) edge monomials of
\(G^{\mathbf n}\) whose product divides \(u\). Hence $u\in I(G^{\mathbf n})^s$.
\end{proof}

\begin{definition}\label{def:defect-antichain}
For \(s\geq 1\), define the \(s\)-th symbolic defect region of \(G\) by
\[
\mathcal E_s(G)=\mathcal P_s(G)\setminus\mathcal O_s(G).
\]
The \(s\)-th symbolic defect antichain of \(G\) is
\[
\mathcal D_s(G)=\min_{\leq}\mathcal E_s(G),
\]
where the minimum is taken with respect to the componentwise order on
\(\mathbb N^r\). Thus \(\mathcal D_s(G)\) consists of the minimal exponent
vectors which satisfy the symbolic-cover inequalities but do not dominate an
\(s\)-fold edge-incidence vector.
\end{definition}

\begin{lemma}\label{lem:defect-antichain-minimal-generators}
Let \(G\) be a finite simple graph, let \(G^{\mathbf n}\) be its
\(\mathbf n\)-blow-up, and set $J=I(G^{\mathbf n})$. A monomial \(u\in S_{\mathbf n}\) represents a minimal monomial generator of $J^{(s)}/J^s$ if and only if $\pi(u)\in\mathcal D_s(G)$.
\end{lemma}

\begin{proof}
By Lemmas~\ref{lem:symbolic-membership-part-degree} and
\ref{lem:ordinary-membership-part-degree}, membership in the symbolic and
ordinary powers of the blow-up ideal is completely determined by the part-degree
vector \(\pi(u)\),
\[
u\in J^{(s)}\setminus J^s
\quad\Longleftrightarrow\quad
\pi(u)\in\mathcal P_s(G)\setminus\mathcal O_s(G)
=
\mathcal E_s(G).
\]
Thus, nonzero monomial classes in \(J^{(s)}/J^s\) are exactly those whose part-degree vectors lie in \(\mathcal E_s(G)\). If \(u\) is minimal in \(J^{(s)}/J^s\), then no proper divisor of \(u\) can have part-degree vector in \(\mathcal E_s(G)\). Hence \(\pi(u)\) is minimal in \(\mathcal E_s(G)\), so $\pi(u)\in\mathcal D_s(G)$.
Conversely, suppose \(\pi(u)\in\mathcal D_s(G)\). Then \(u\in J^{(s)}\setminus J^s\). If a proper divisor \(w\mid u\) also represented a nonzero class in
\(J^{(s)}/J^s\), then $\pi(w)\in\mathcal E_s(G)$ and $\pi(w)\leq\pi(u)$.
By minimality of \(\pi(u)\), we would have \(\pi(w)=\pi(u)\). Since \(w\mid u\)
and both monomials have the same total degree in every part \(X_i\), this
forces \(w=u\), a contradiction. Hence \(u\) is minimal.
\end{proof}

\begin{theorem}[Blow-up transfer formula]\label{thm:blowup-transfer-symbolic-defect}
Let \(G\) be a finite simple graph with vertex set $V(G)=\{v_1,\ldots,v_r\}$, and let \(G^{\mathbf n}\) be its \(\mathbf n\)-blow-up. Set $J=I(G^{\mathbf n})\subseteq S_{\mathbf n}$. Then, for every \(s\geq 1\),
\[
\operatorname{sdefect}(J,s)
=
\sum_{\mathbf a\in\mathcal D_s(G)}
\prod_{i=1}^{r}
\binom{a_i+n_i-1}{n_i-1}.
\]
\end{theorem}

\begin{proof}
By Lemma~\ref{lem:defect-antichain-minimal-generators}, the minimal monomial generators of \(J^{(s)}/J^s\) are precisely the monomials whose part-degree vectors lie in \(\mathcal D_s(G)\). Fix $\mathbf a=(a_1,\ldots,a_r)\in\mathcal D_s(G)$. The number of monomials in the variables of \(X_i\) with total degree \(a_i\) is
\[
\binom{a_i+n_i-1}{n_i-1}.
\]
Since the choices in different parts are independent, the number of monomials
with part-degree vector \(\mathbf a\) is
\[
\prod_{i=1}^{r}
\binom{a_i+n_i-1}{n_i-1}.
\]
Summing over all \(\mathbf a\in\mathcal D_s(G)\) gives the formula.
\end{proof}

\begin{remark}
Theorem~\ref{thm:blowup-transfer-symbolic-defect} applies to the blow-up
\(G^{\mathbf n}\) of an arbitrary finite simple graph \(G\). Thus the ambient
class considered in this paper is not a fixed graph family, but the full
blow-up closure of all finite simple graphs. Complete multipartite graphs,
complete split graphs, and blow-ups of odd cycles are only special cases
obtained by choosing different base graphs.
\end{remark}

\subsection{Defect antichains as obstruction data}
\label{subsec:defect-antichains-obstruction-data}

The symbolic defect antichain should be viewed as obstruction data rather than only as a counting device. For a fixed symbolic degree \(s\), the region $\mathcal P_s(G)$ records the exponent vectors satisfying the symbolic-cover inequalities, while $\mathcal O_s(G)$ records the exponent vectors which dominate an \(s\)-fold product of edge
incidence vectors. Hence, the difference
\[
\mathcal E_s(G)=\mathcal P_s(G)\setminus\mathcal O_s(G)
\]
measures the failure of symbolic and ordinary powers to agree at the level of
exponent vectors.
The antichain $\mathcal D_s(G)=\min_{\leq}\mathcal E_s(G)$ therefore consists of the minimal exponent obstructions responsible for the quotient $I(G)^{(s)}/I(G)^s$.
By Dickson's lemma, \(\mathcal D_s(G)\) is finite.

\begin{proposition}\label{prop:base-defect-antichain-obstruction}
Let \(G\) be a finite simple graph and set $I=I(G)$. Then the minimal monomial generators of $I^{(s)}/I^s$ are precisely the monomials $x^{\mathbf a}=x_1^{a_1}\cdots x_r^{a_r}$ with $\mathbf a\in\mathcal D_s(G)$.
In particular,
\begin{equation}\label{eqsan2}
    \operatorname{sdefect}(I(G),s)=|\mathcal D_s(G)|.
\end{equation}
\end{proposition}

\begin{proof}
This is the special case of Lemma~\ref{lem:defect-antichain-minimal-generators}
with $\mathbf n=(1,\ldots,1)$. In this case every part contains one variable, so the part-degree vector of a
monomial is exactly its exponent vector. Thus the minimal monomial generators of $I^{(s)}/I^s$ are indexed by
$\mathcal D_s(G)$. Counting these generators gives Eq.\eqref{eqsan2}.
\end{proof}

For a blow-up \(G^{\mathbf n}\), the same obstruction vector $\mathbf a=(a_1,\ldots,a_r)\in\mathcal D_s(G)$
does not give only one monomial. Instead, it expands into all monomials whose total degree in the blown-up part \(X_i\) is \(a_i\) for every \(i\). The number of such monomials is
\[
W_{\mathbf n}(\mathbf a)
=
\prod_{i=1}^{r}
\binom{a_i+n_i-1}{n_i-1}.
\]
Thus the blow-up transfer formula can be written as the weighted antichain sum
\[
\operatorname{sdefect}(I(G^{\mathbf n}),s)
=
\sum_{\mathbf a\in\mathcal D_s(G)}
W_{\mathbf n}(\mathbf a).
\]
This shows that the symbolic defect function of a blow-up is controlled by two separate pieces of data: the base obstruction set \(\mathcal D_s(G)\), which depends only on \(G\) and \(s\), and the weight $W_{\mathbf n}(\mathbf a)$,
which depends only on the blow-up sizes. In the next section, this numerical weight is refined to a multigraded generator series by replacing the binomial factor with complete homogeneous symmetric polynomials.
\begin{example}
Let \(G=K_3\) and \(s=2\). Then the minimal vertex covers of \(K_3\) are
\(\{v_1,v_2\}\), \(\{v_1,v_3\}\), and \(\{v_2,v_3\}\). Hence
\[
\mathcal P_2(K_3)=
\left\{
(a_1,a_2,a_3)\in\mathbb N^3:
a_1+a_2\geq 2,\ 
a_1+a_3\geq 2,\ 
a_2+a_3\geq 2
\right\}.
\]
The vector \((1,1,1)\) belongs to \(\mathcal P_2(K_3)\), but it does not belong to \(\mathcal O_2(K_3)\), since a product of two edges of \(K_3\) has total degree \(4\). Moreover, every proper componentwise smaller vector fails one of the cover inequalities. Therefore $\mathcal D_2(K_3)=\{(1,1,1)\}$.
Consequently, $\operatorname{sdefect}(I(K_3),2)=1$.
\end{example}
\subsection{Effective computation and semigroup interpretation}

We next record an effective consequence of the antichain construction. Although the region
\(\mathcal P_s(G)\setminus \mathcal O_s(G)\) may be infinite, its minimal elements lie in a finite box.

\begin{lemma}
Let \(G\) be a finite simple graph with vertex set \(V(G)=\{v_1,\ldots,v_r\}\). Then, for every \(s\geq 1\), $\mathcal D_s(G)\subseteq \{0,1,\ldots,s\}^r$.
\end{lemma}

\begin{proof}
Let \(\mathbf a=(a_1,\ldots,a_r)\in \mathcal D_s(G)\). Suppose that \(a_i\geq s+1\) for some \(i\). We claim that \(\mathbf a-\mathbf e_i\in \mathcal P_s(G)\). Indeed, if \(C\) is a minimal vertex cover not containing \(v_i\), then the corresponding cover inequality is unchanged. If \(v_i\in C\), then
\[
\sum_{v_j\in C} a_j \geq a_i \geq s+1,
\]
and hence after subtracting \(1\) from the \(i\)-th coordinate the sum is still at least \(s\). Thus
\(\mathbf a-\mathbf e_i\in \mathcal P_s(G)\).
If \(\mathbf a-\mathbf e_i\in \mathcal O_s(G)\), then since \(\mathcal O_s(G)\) is upward closed under the componentwise order, we would have \(\mathbf a\in \mathcal O_s(G)\), contradicting \(\mathbf a\in \mathcal P_s(G)\setminus\mathcal O_s(G)\). Therefore $\mathbf a-\mathbf e_i\in \mathcal P_s(G)\setminus\mathcal O_s(G)$, which contradicts the minimality of \(\mathbf a\). Hence, every coordinate of \(\mathbf a\) is at most \(s\).
\end{proof}

This gives the following finite procedure.

\begin{algorithm}
\caption{Finite procedure for computing the symbolic defect antichain \(\mathcal D_s(G)\)}
\label{alg:defect-antichain}
\begin{enumerate}
\item Determine the set \(\mathcal C(G)\) of minimal vertex covers of \(G\).

\item Form the symbolic exponent region
\[
\mathcal P_s(G)=
\left\{
\mathbf a\in\mathbb{N}^r \;:\;
\sum_{v_i\in C} a_i \geq s
\ \text{for every } C\in\mathcal C(G)
\right\}.
\]

\item Form the ordinary exponent region
\[
\mathcal O_s(G)=
\left\{
\mathbf a\in\mathbb{N}^r \;:\;
\mathbf a\geq A_G\boldsymbol\lambda
\ \text{for some } \boldsymbol\lambda\in\mathbb{N}^{E(G)}
\text{ with } \sum_{e\in E(G)}\lambda_e=s
\right\}.
\]

\item Search the finite box \(\{0,1,\ldots,s\}^r\), retain the vectors in
\(\mathcal P_s(G)\setminus\mathcal O_s(G)\), and return the componentwise
minimal elements.
\end{enumerate}
\end{algorithm}

\begin{proposition}
Algorithm~\ref{alg:defect-antichain} terminates and returns \(\mathcal D_s(G)\).
\end{proposition}

\begin{proof}
The algorithm terminates because it searches the finite box \(\{0,1,\ldots,s\}^r\). By the preceding lemma, every element of \(\mathcal D_s(G)\) lies in this box. The algorithm keeps precisely those vectors in \(\mathcal P_s(G)\setminus \mathcal O_s(G)\) and then takes the minimal elements with respect to the componentwise order. This is exactly the definition of $\mathcal D_s(G)=\min_{\leq}\bigl(\mathcal P_s(G)\setminus\mathcal O_s(G)\bigr)$.
\end{proof}

The construction also has a natural semigroup interpretation. Define
\[
\mathcal S(G)=
\left\{
(\mathbf a,s)\in\mathbb N^r\times\mathbb N:
\mathbf a\in\mathcal P_s(G)
\right\}
\]
and
\[
\mathcal R(G)=
\left\{
(\mathbf a,s)\in\mathbb N^r\times\mathbb N:
\mathbf a\in\mathcal O_s(G)
\right\}.
\]
Then \(\mathcal S(G)\) is the exponent semigroup of the symbolic Rees algebra of \(I(G)\), while \(\mathcal R(G)\) is the exponent semigroup of the ordinary Rees algebra of \(I(G)\), after allowing multiplication by monomials of the base ring. Thus, for each fixed \(s\), the antichain \(\mathcal D_s(G)\) is the set of componentwise minimal elements in the \(s\)-th fiber of $\mathcal S(G)\setminus \mathcal R(G)$.
Finally, the multigraded defect series can be interpreted as a Hilbert series of minimal generators. Let \(J=I(G^{\mathbf n})\subseteq S_{\mathbf n}\), let $Q_{J,s}=J^{(s)}/J^s$,
and let \(\mathfrak m_{\mathbf n}\) be the homogeneous maximal ideal of
\(S_{\mathbf n}\). Then $Q_{J,s}/\mathfrak m_{\mathbf n}Q_{J,s}$ is a finite-dimensional multigraded \(K\)-vector space whose homogeneous basis
is given by the minimal monomial generators of \(J^{(s)}/J^s\). Hence
$\mathcal M_{J,s}(\mathbf z)=H_{Q_{J,s}/\mathfrak m_{\mathbf n}Q_{J,s}}(\mathbf z)$, where \(H\) denotes the multigraded Hilbert series. In particular,
\[
\operatorname{sdefect}(J,s)
=
\dim_K\bigl(Q_{J,s}/\mathfrak m_{\mathbf n}Q_{J,s}\bigr)
=
\mathcal M_{J,s}(1,\ldots,1).
\]

\section{Multigraded defect series of graph blow-ups}

\label{sec:multigraded-defect-series}
In this section, we refine the numerical symbolic defect into a multigraded defect-generator series. This refinement shows that the blow-up transfer theorem does not merely count generators; it describes their full multigraded distribution. The numerical symbolic defect is then recovered as a specialization of this series. Let \(G\) be a finite simple graph with vertex set $V(G)=\{v_1,\ldots,v_r\}$, and let $\mathbf n=(n_1,\ldots,n_r)\in\mathbb N^r$. Let \(G^{\mathbf n}\) be the \(\mathbf n\)-blow-up of \(G\), and set
\[
J=I(G^{\mathbf n})\subseteq
S_{\mathbf n}=K[x_{ij}:1\leq i\leq r,\ 1\leq j\leq n_i].
\]
For \(s\geq 1\), let $\mathcal G_s(J)$ denote the set of monomials whose residue classes form the minimal monomial
generators of the quotient $J^{(s)}/J^s$. Equivalently,
\[
\mathcal G_s(J)
=
\left\{
u\in J^{(s)}\setminus J^s:
\text{ no proper divisor of }u\text{ belongs to }J^{(s)}\setminus J^s
\right\}.
\]

For a monomial
\[
u=\prod_{i=1}^{r}\prod_{j=1}^{n_i}x_{ij}^{b_{ij}},
\]
we write
\[
\mathbf z^u
=
\prod_{i=1}^{r}\prod_{j=1}^{n_i}z_{ij}^{b_{ij}},
\]
where $\mathbf z=\{z_{ij}:1\leq i\leq r,\ 1\leq j\leq n_i\}$ is a set of auxiliary variables. We define the multigraded defect series of
\(J\) in symbolic degree \(s\) by
\[
\mathcal M_{J,s}(\mathbf z)
=
\sum_{u\in\mathcal G_s(J)}
\mathbf z^u.
\]

For \(d\geq 0\), let $h_d(z_{i1},\ldots,z_{in_i})$ denote the complete homogeneous symmetric polynomial of degree \(d\) in the variables \(z_{i1},\ldots,z_{in_i}\); that is,
\[
h_d(z_{i1},\ldots,z_{in_i})
=
\sum_{\substack{\beta_1+\cdots+\beta_{n_i}=d\\ \beta_j\geq 0}}
z_{i1}^{\beta_1}\cdots z_{in_i}^{\beta_{n_i}}.
\]
We use the convention $h_0(z_{i1},\ldots,z_{in_i})=1$.

\begin{theorem}[Multigraded defect series of a blow-up]
\label{thm:multigraded-defect-series}
Let \(G\) be a finite simple graph and let \(G^{\mathbf n}\) be its \(\mathbf n\)-blow-up. Set $J=I(G^{\mathbf n})$.
Then, for every \(s\geq 1\),
\begin{equation}\label{eqsan3}
\mathcal M_{J,s}(\mathbf z)
=
\sum_{\mathbf a\in\mathcal D_s(G)}
\prod_{i=1}^{r}
h_{a_i}(z_{i1},\ldots,z_{in_i}),
\end{equation}
where $\mathbf a=(a_1,\ldots,a_r)$.
\end{theorem}

\begin{proof}
By Lemma~\ref{lem:defect-antichain-minimal-generators}, a monomial $u\in S_{\mathbf n}$ belongs to \(\mathcal G_s(J)\) if and only if its part-degree vector satisfies $\pi(u)\in\mathcal D_s(G)$. Therefore,
\[
\mathcal M_{J,s}(\mathbf z)
=
\sum_{\mathbf a\in\mathcal D_s(G)}
\sum_{\substack{u\in S_{\mathbf n}\\ \pi(u)=\mathbf a}}
\mathbf z^u.
\]

Fix $\mathbf a=(a_1,\ldots,a_r)\in\mathcal D_s(G)$. The condition $\pi(u)=\mathbf a$ means that, for each \(i\), the total degree of \(u\) in the variables $x_{i1},\ldots,x_{in_i}$ is exactly \(a_i\). Hence, the contribution from the \(i\)-th part is precisely $h_{a_i}(z_{i1},\ldots,z_{in_i})$. Since the choices in different parts are independent, we obtain
\[
\sum_{\substack{u\in S_{\mathbf n}\\ \pi(u)=\mathbf a}}
\mathbf z^u
=
\prod_{i=1}^{r}
h_{a_i}(z_{i1},\ldots,z_{in_i}).
\]
Summing over all $\mathbf a\in\mathcal D_s(G)$ gives Eq. \eqref{eqsan3} as required.
\end{proof}

\begin{corollary}[Numerical specialization]
\label{cor:numerical-specialization-defect-series}
Let \(J=I(G^{\mathbf n})\). Then $\operatorname{sdefect}(J,s)=\mathcal M_{J,s}(1,\ldots,1)$.

Equivalently,
\[
\operatorname{sdefect}(J,s)
=
\sum_{\mathbf a\in\mathcal D_s(G)}
\prod_{i=1}^{r}
\binom{a_i+n_i-1}{n_i-1}.
\]
\end{corollary}

\begin{proof}
The value $\mathcal M_{J,s}(1,\ldots,1)$ counts the number of monomials in \(\mathcal G_s(J)\). Since
\(\mathcal G_s(J)\) is the set of minimal monomial generators of $J^{(s)}/J^s$, this number is $\mu(J^{(s)}/J^s)=
\operatorname{sdefect}(J,s)$.
Moreover,
\[
h_{a_i}(1,\ldots,1)
=
\binom{a_i+n_i-1}{n_i-1}.
\]
The formula follows from Theorem~\ref{thm:multigraded-defect-series}.
\end{proof}

Theorem~\ref{thm:multigraded-defect-series} also has a coarser specialization which records only part-degree vectors rather than the individual variables inside each part. Introduce variables $y_1,\ldots,y_r$ corresponding to the parts \(X_1,\ldots,X_r\), and set $z_{i1}=\cdots=z_{in_i}=y_i$.

\begin{corollary}[Part-degree defect series]
\label{cor:part-degree-defect-series}
Let \(J=I(G^{\mathbf n})\). Define
\[
\mathcal P\mathcal M_{J,s}(y_1,\ldots,y_r)
=
\mathcal M_{J,s}(\mathbf z)\big|_{z_{i1}=\cdots=z_{in_i}=y_i}.
\]
Then
\[
\mathcal P\mathcal M_{J,s}(y_1,\ldots,y_r)
=
\sum_{\mathbf a\in\mathcal D_s(G)}
\left(
\prod_{i=1}^{r}
\binom{a_i+n_i-1}{n_i-1}
\right)
y_1^{a_1}\cdots y_r^{a_r}.
\]
\end{corollary}

\begin{proof}
By Theorem~\ref{thm:multigraded-defect-series},
\[
\mathcal M_{J,s}(\mathbf z)
=
\sum_{\mathbf a\in\mathcal D_s(G)}
\prod_{i=1}^{r}
h_{a_i}(z_{i1},\ldots,z_{in_i}).
\]
After setting
\[
z_{i1}=\cdots=z_{in_i}=y_i,
\]
we have
\[
h_{a_i}(y_i,\ldots,y_i)
=
\binom{a_i+n_i-1}{n_i-1}y_i^{a_i}.
\]
Therefore
\[
\mathcal P\mathcal M_{J,s}(y_1,\ldots,y_r)
=
\sum_{\mathbf a\in\mathcal D_s(G)}
\left(
\prod_{i=1}^{r}
\binom{a_i+n_i-1}{n_i-1}
\right)
y_1^{a_1}\cdots y_r^{a_r}.
\]
\end{proof}

\begin{corollary}[Base graph specialization]
\label{cor:base-graph-antichain-enumerator}
Let $\mathbf n=(1,\ldots,1)$. Then $G^{\mathbf n}=G$. If \(I=I(G)\), then
\[
\mathcal M_{I,s}(z_1,\ldots,z_r)
=
\sum_{\mathbf a\in\mathcal D_s(G)}
z_1^{a_1}\cdots z_r^{a_r}.
\]
In particular, $\operatorname{sdefect}(I(G),s)=|\mathcal D_s(G)|$.
\end{corollary}

\begin{proof}
When $\mathbf n=(1,\ldots,1)$, each part contains exactly one variable. Therefore $h_{a_i}(z_i)=z_i^{a_i}$.
The formula follows immediately from
Theorem~\ref{thm:multigraded-defect-series}. Setting $z_1=\cdots=z_r=1$ gives $\operatorname{sdefect}(I(G),s)=|\mathcal D_s(G)|$.
\end{proof}

\begin{remark}
The multigraded defect series shows that the symbolic defect antichain is a finer invariant than the numerical symbolic defect. The number $\operatorname{sdefect}(J,s)$ only records the cardinality of the minimal defect generators. By contrast, $\mathcal M_{J,s}(\mathbf z)$ records their full multigraded distribution. Thus, the blow-up transfer theorem is not merely a stars-and-bars count; it is the numerical specialization of a
multigraded antichain formula.
\end{remark}

\begin{corollary}[Complete multipartite specialization]
\label{cor:complete-multipartite-defect-series}
Let $J=I(K_{n_1,\ldots,n_t})$. Then, for every \(s\geq 1\),
\[
\mathcal M_{J,s}(\mathbf z)
=
\sum_{\mathbf a\in\mathcal D_s(K_t)}
\prod_{i=1}^{t}
h_{a_i}(z_{i1},\ldots,z_{in_i}).
\]
Using the partition classification of $\mathcal D_s(K_t)$, this becomes
\[
\mathcal M_{J,s}(\mathbf z)
=
\sum_{p=1}^{s-1}
\ \sum_{\lambda\in\Lambda_{s,p}^{(t)}}
\ \sum_{\mathbf a\in\operatorname{Orb}_t(\lambda)}
\prod_{i=1}^{t}
h_{a_i}(z_{i1},\ldots,z_{in_i}).
\]
\end{corollary}

\begin{proof}
Since $K_{n_1,\ldots,n_t}=K_t^{\mathbf n}$, the first formula follows from Theorem~\ref{thm:multigraded-defect-series} applied to the base graph \(K_t\). The second formula follows by substituting the partition classification of
\(\mathcal D_s(K_t)\).
\end{proof}

\begin{remark}
Corollary~\ref{cor:complete-multipartite-defect-series} upgrades the numerical
symbolic defect formula for complete multipartite graphs to a multigraded
formula for the minimal defect generators. The numerical formula is recovered
by setting all auxiliary variables equal to \(1\).
\end{remark}

\section{Partition classification for complete graphs}
\label{sec:partition-complete-graphs}

In this section we classify the symbolic defect antichains of complete graphs in
all symbolic degrees. This classification is useful because every complete
multipartite graph is a blow-up of a complete graph. Thus the results of this
section provide the base antichain data needed for the applications in the next
section.
Let $K_t$ be the complete graph on the vertex set $V(K_t)=\{v_1,\ldots,v_t\}$. For a vector $\mathbf a=(a_1,\ldots,a_t)\in\mathbb N^t$, write $|\mathbf a|=a_1+\cdots+a_t$ and $\max(\mathbf a)=\max\{a_1,\ldots,a_t\}$.

\subsection{Symbolic and ordinary regions of \(K_t\)}

We first record explicit descriptions of the symbolic and ordinary exponent
regions of \(K_t\).

\begin{lemma}\label{lem:complete-graph-regions}
Let \(t\geq 2\) and \(s\geq 1\). Then
\[
\mathcal P_s(K_t)
=
\left\{
\mathbf a\in\mathbb N^t:
|\mathbf a|-a_i\geq s
\text{ for every } i=1,\ldots,t
\right\}.
\]
Moreover,
\[
\mathcal O_s(K_t)
=
\left\{
\mathbf a\in\mathbb N^t:
\sum_{i=1}^{t}\min\{a_i,s\}\geq 2s
\right\}.
\]
\end{lemma}

\begin{proof}
The minimal vertex covers of \(K_t\) are precisely the sets $V(K_t)\setminus\{v_i\}$, $i=1,\ldots,t$. Hence \(\mathbf a\in\mathcal P_s(K_t)\) if and only if $\sum_{j\neq i}a_j\geq s$ for every \(i\), which is equivalent to $|\mathbf a|-a_i\geq s$ for every \(i\).
We now prove the description of \(\mathcal O_s(K_t)\). If
\(\mathbf a\in\mathcal O_s(K_t)\), then \(\mathbf a\) dominates the degree
vector of a product of \(s\) edge monomials of \(K_t\). Such a product uses
exactly \(2s\) vertex occurrences, and no vertex can occur more than \(s\)
times. Therefore
\[
\sum_{i=1}^{t}\min\{a_i,s\}\geq 2s.
\]

Conversely, suppose that $\sum_{i=1}^{t}\min\{a_i,s\}\geq 2s$. Choose integers \(c_i\) such that $0\leq c_i\leq \min\{a_i,s\}$ for all \(i\), and $c_1+\cdots+c_t=2s$. Then \(c_i\leq s\) for all \(i\). We claim that
\(\mathbf c=(c_1,\ldots,c_t)\) is the degree vector of a multiset of \(s\)
edges of \(K_t\). This is clear for \(s=1\). For \(s>1\), choose two positive coordinates \(c_p,c_q\) so that every coordinate equal to \(s\) is among \(c_p,c_q\). Set $\mathbf c'=\mathbf c-\mathbf e_p-\mathbf e_q$. Then $|\mathbf c'|=2(s-1)$ and every coordinate of \(\mathbf c'\) is at most \(s-1\). By induction,
\(\mathbf c'\) is the degree vector of a multiset of \(s-1\) edges of \(K_t\).
Adding the edge \(\{v_p,v_q\}\) gives a multiset of \(s\) edges with degree vector \(\mathbf c\).
Since \(\mathbf c\leq\mathbf a\), the monomial with exponent vector
\(\mathbf a\) is divisible by a product of \(s\) edge monomials of \(K_t\).
Thus $\mathbf a\in\mathcal O_s(K_t)$.
\end{proof}

\subsection{Partition types}

We now introduce the partition notation used in the classification. For integers \(s\geq 1\) and \(1\leq p\leq s-1\), let $\Lambda_{s,p}^{(t)}$ denote the set of partitions $\lambda=(\lambda_1,\ldots,\lambda_m)$ such that $2\leq m\leq t$,  $\lambda_1\geq\lambda_2\geq\cdots\geq\lambda_m\geq 1$, $\lambda_1=p$, $|\lambda|=\lambda_1+\cdots+\lambda_m=s+p$, and the largest part \(p\) occurs at least twice. Equivalently, $\lambda_1=\lambda_2=p$. For such a partition \(\lambda\), let $\operatorname{Orb}_t(\lambda)$ be the set of all vectors in \(\mathbb N^t\) obtained by placing the parts of \(\lambda\) in \(m\) coordinates and filling the remaining \(t-m\) coordinates
with zeros, in all distinct ways.

\subsection{Classification theorem}

\begin{theorem}[Partition classification of \(\mathcal D_s(K_t)\)]
\label{thm:partition-classification-complete-graph}
Let \(t\geq 2\) and \(s\geq 1\). Then
\[
\mathcal D_s(K_t)
=
\bigcup_{p=1}^{s-1}
\ \bigcup_{\lambda\in\Lambda_{s,p}^{(t)}}
\operatorname{Orb}_t(\lambda).
\]
Equivalently, a vector \(\mathbf a\in\mathbb N^t\) belongs to \(\mathcal D_s(K_t)\) if and only if there exists an integer $p\in\{1,\ldots,s-1\}$ such that $|\mathbf a|=s+p$, $\max(\mathbf a)=p$, and the maximum value \(p\) occurs at least twice among the coordinates of
\(\mathbf a\).
\end{theorem}

\begin{proof}
Let $\mathbf a=(a_1,\ldots,a_t)\in\mathcal D_s(K_t)$. Set $d=|\mathbf a|$ and $L=\max(\mathbf a)$. Since $\mathbf a\in\mathcal P_s(K_t)$, Lemma~\ref{lem:complete-graph-regions} gives $d-a_i\geq s$ for every \(i\). Hence $L\leq d-s$.
We first show that every coordinate of \(\mathbf a\) is at most \(s-1\). Suppose that \(a_i\geq s\) for some \(i\). Since \(\mathbf a\in\mathcal P_s(K_t)\), we have $d-a_i\geq s$.
Therefore, the total degree outside the \(i\)-th coordinate is at least \(s\).
It follows that
\[
\sum_{j=1}^{t}\min\{a_j,s\}\geq s+s=2s.
\]
By Lemma~\ref{lem:complete-graph-regions}, this implies $\mathbf a\in\mathcal O_s(K_t)$, contradicting
$\mathbf a\in\mathcal D_s(K_t)\subseteq\mathcal P_s(K_t)\setminus\mathcal O_s(K_t)$.
Thus $a_i\leq s-1$ for every \(i\).
Consequently,
\[
\sum_{i=1}^{t}\min\{a_i,s\}=|\mathbf a|=d.
\]
Since $\mathbf a\notin\mathcal O_s(K_t)$, Lemma~\ref{lem:complete-graph-regions} gives $d<2s$. Thus $d\leq 2s-1$.
We now prove that $L=d-s$. We already have \(L\leq d-s\). Suppose, for contradiction, that $L<d-s$. Choose an index \(j\) with \(a_j>0\), and set $\mathbf b=\mathbf a-\mathbf e_j$. Then $|\mathbf b|=d-1$ and $\max(\mathbf b)\leq L\leq d-s-1=(d-1)-s$.
Hence $|\mathbf b|-\max(\mathbf b)\geq s$, so $\mathbf b\in\mathcal P_s(K_t)$. Moreover, all coordinates of \(\mathbf b\) are at most \(s-1\), and $|\mathbf b|=d-1<2s$. Therefore
\[
\sum_{i=1}^{t}\min\{b_i,s\}=|\mathbf b|<2s,
\]
so $\mathbf b\notin\mathcal O_s(K_t)$. Thus $\mathbf b\in\mathcal P_s(K_t)\setminus\mathcal O_s(K_t)$, contradicting the minimality of \(\mathbf a\). Hence $L=d-s$.
Next we show that the maximum value \(L\) occurs at least twice. Suppose that \(L\) occurs only once, say at the \(j\)-th coordinate. Set $\mathbf b=\mathbf a-\mathbf e_j$. Then $|\mathbf b|=d-1$ and
$\max(\mathbf b)\leq L-1=(d-s)-1=(d-1)-s$.
Hence $\mathbf b\in\mathcal P_s(K_t)$. As before, $|\mathbf b|=d-1<2s$ and all coordinates of \(\mathbf b\) are at most \(s-1\), so $\mathbf b\notin\mathcal O_s(K_t)$. This again contradicts the minimality of \(\mathbf a\). Therefore, the maximum value \(L\) occurs at least twice.
Now set $p=L$. Since $L=d-s$, we have $d=s+p$. Also, because \(d\leq 2s-1\), we get $p\leq s-1$. Since \(L>0\), we have $p\geq 1$. Therefore $p\in\{1,\ldots,s-1\}$. The nonzero coordinates of \(\mathbf a\), written in nonincreasing order, form a partition $\lambda\in\Lambda_{s,p}^{(t)}$. Hence
\[
\mathbf a\in
\bigcup_{p=1}^{s-1}
\ \bigcup_{\lambda\in\Lambda_{s,p}^{(t)}}
\operatorname{Orb}_t(\lambda).
\]
Conversely, suppose that $\mathbf a\in\operatorname{Orb}_t(\lambda)$ for some $\lambda\in\Lambda_{s,p}^{(t)}$ and some \(p\in\{1,\ldots,s-1\}\). Then $|\mathbf a|=s+p$ and $\max(\mathbf a)=p$. Therefore $|\mathbf a|-\max(\mathbf a)=s$.
Hence $|\mathbf a|-a_i\geq s$ for every \(i\), and so $\mathbf a\in\mathcal P_s(K_t)$. Since \(p\leq s-1\), every coordinate of \(\mathbf a\) is at most \(s-1\).
Thus $\sum_{i=1}^{t}\min\{a_i,s\}=|\mathbf a|=s+p<2s$. By Lemma~\ref{lem:complete-graph-regions}, $\mathbf a\notin\mathcal O_s(K_t)$. Hence
$\mathbf a\in\mathcal P_s(K_t)\setminus\mathcal O_s(K_t)$.
It remains to prove minimality. Let $\mathbf 0\leq\mathbf b<\mathbf a$. Set $R=|\mathbf a|-|\mathbf b|>0$.
Let $M=\{i:a_i=p\}$ be the set of coordinates where \(\mathbf a\) attains its maximum. By
assumption, \(|M|\geq 2\). For \(i\in M\), write $\delta_i=a_i-b_i$. Since $\sum_{i\in M}\delta_i\leq R$ and \(|M|\geq 2\), there exists \(i_0\in M\) such that $\delta_{i_0}<R$. Therefore $b_{i_0}=p-\delta_{i_0}>p-R$.
But $|\mathbf b|=|\mathbf a|-R=s+p-R$, so $|\mathbf b|-s=p-R$. Thus $\max(\mathbf b)>|\mathbf b|-s$. Equivalently,
$|\mathbf b|-\max(\mathbf b)<s$. By Lemma~\ref{lem:complete-graph-regions}, $\mathbf b\notin\mathcal P_s(K_t)$. Therefore, no proper componentwise smaller vector lies in $\mathcal P_s(K_t)\setminus\mathcal O_s(K_t)$.
Hence $\mathbf a\in\mathcal D_s(K_t)$.

This proves the classification.
\end{proof}

\subsection{Low-degree consequences}

We record the first cases explicitly. These will be used in the applications
and examples.

\begin{corollary}\label{cor:D2-D3-D4-complete-graph}
Let \(t\geq 2\). Then $\mathcal D_2(K_t)
=
\left\{
\mathbf e_i+\mathbf e_j+\mathbf e_k:
1\leq i<j<k\leq t
\right\}$.
Moreover, $\mathcal D_3(K_t)
=
\mathcal A_3(t)\cup\mathcal B_3(t)$,
where
\[
\mathcal A_3(t)
=
\left\{
\mathbf e_i+\mathbf e_j+\mathbf e_k+\mathbf e_\ell:
1\leq i<j<k<\ell\leq t
\right\}
\]
and
\[
\mathcal B_3(t)
=
\left\{
2\mathbf e_i+2\mathbf e_j+\mathbf e_k:
i,j,k\text{ are distinct and } i<j
\right\}.
\]
Finally, \(\mathcal D_4(K_t)\) consists of the orbits of the partition patterns
\[
(1,1,1,1,1),
\qquad
(2,2,2),
\qquad
(2,2,1,1),
\qquad
(3,3,1),
\]
with at most \(t\) parts.
\end{corollary}

\begin{proof}
For \(s=2\), the only possible value is \(p=1\). The corresponding partition
has total degree \(3\), largest part \(1\), and largest part occurring at least
twice. Thus, the only partition is $(1,1,1)$, which gives the stated description of \(\mathcal D_2(K_t)\).

For \(s=3\), if \(p=1\), the partition has total degree \(4\) and largest part \(1\), giving $(1,1,1,1)$.
If \(p=2\), the partition has total degree \(5\), largest part \(2\), and the
largest part must occur at least twice, giving $(2,2,1)$. This yields the stated description of \(\mathcal D_3(K_t)\).

For \(s=4\), the possible values are \(p=1,2,3\). If \(p=1\), the partition is $(1,1,1,1,1)$. If \(p=2\), the total degree is \(6\), and the possible partition types with
largest part \(2\) occurring at least twice are $(2,2,2)$ and $(2,2,1,1)$.
If \(p=3\), the total degree is \(7\), and the only possible partition type is $(3,3,1)$. The result follows from
Theorem~\ref{thm:partition-classification-complete-graph}.
\end{proof}

\begin{remark}
Theorem~\ref{thm:partition-classification-complete-graph} shows that the
complete graph case is controlled by partition types rather than by the sizes of
the blow-up parts. The blow-up parameters enter only later through the weights
in Theorem~\ref{thm:blowup-transfer-symbolic-defect}. This separation is what
allows complete multipartite graph formulas to be derived uniformly in all
symbolic degrees.
\end{remark}

\section{Applications of the general blow-up transfer formula}
\label{sec:applications-blowup}

The following examples are not separate targets of the paper. They illustrate
how the general theorem applies to different base graphs. The ambient class is
the class of all blow-ups \(G^{\mathbf n}\) of arbitrary finite simple graphs.

\subsection{Vanishing transfer}

We first record that vanishing of symbolic defects transfers from a base graph
to all of its blow-ups.

\begin{corollary}\label{cor:vanishing-transfer}
Let \(G\) be a finite simple graph and let \(G^{\mathbf n}\) be its \(\mathbf n\)-blow-up. Fix \(s\geq 1\). If $I(G)^{(s)}=I(G)^s$, then $I(G^{\mathbf n})^{(s)}=I(G^{\mathbf n})^s$.
In particular, if \(G\) is bipartite, then $I(G^{\mathbf n})^{(s)}=I(G^{\mathbf n})^s$ for all \(s\geq 1\).
\end{corollary}

\begin{proof}
If \(I(G)^{(s)}=I(G)^s\), then $\mathcal D_s(G)=\emptyset$. Hence, by Theorem~\ref{thm:blowup-transfer-symbolic-defect}, $\operatorname{sdefect}(I(G^{\mathbf n}),s)=0$. Therefore $I(G^{\mathbf n})^{(s)}=I(G^{\mathbf n})^s$.
The final assertion follows from the classical fact that edge ideals of
bipartite graphs are normally torsion free; equivalently,
\(I(G)^{(s)}=I(G)^s\) for all \(s\geq 1\) when \(G\) is bipartite
\cite{SIM}.
\end{proof}

\subsection{Complete multipartite graphs}

Complete multipartite graphs arise naturally as blow-ups of complete graphs. Indeed, if $\mathbf n=(n_1,\ldots,n_t)$, then $K_t^{\mathbf n}=K_{n_1,\ldots,n_t}$. Thus, complete multipartite graphs are treated here as applications of the
general blow-up transfer formula rather than as an isolated family.

\begin{corollary}\label{cor:complete-multipartite-general}
Let $J=I(K_{n_1,\ldots,n_t})$. Then, for every \(s\geq 1\),
\[
\operatorname{sdefect}(J,s)
=
\sum_{\mathbf a\in\mathcal D_s(K_t)}
\prod_{i=1}^{t}
\binom{a_i+n_i-1}{n_i-1}.
\]
Equivalently, using the partition classification of
\(\mathcal D_s(K_t)\),
\[
\operatorname{sdefect}(J,s)
=
\sum_{p=1}^{s-1}
\sum_{\lambda\in\Lambda_{s,p}^{(t)}}
\sum_{\mathbf a\in\operatorname{Orb}_t(\lambda)}
\prod_{i=1}^{t}
\binom{a_i+n_i-1}{n_i-1}.
\]
\end{corollary}

\begin{proof}
Since $K_{n_1,\ldots,n_t}=K_t^{(n_1,\ldots,n_t)}$, the first formula follows directly from
Theorem~\ref{thm:blowup-transfer-symbolic-defect}. The second formula follows
by substituting the partition classification of
\(\mathcal D_s(K_t)\).
\end{proof}

The next corollary records the first two nonzero cases explicitly.

\begin{corollary}\label{cor:complete-multipartite-s2-s3}
Let $J=I(K_{n_1,\ldots,n_t})$. Then
\[
\operatorname{sdefect}(J,2)
=
\sum_{1\leq i<j<k\leq t} n_i n_j n_k,
\]
and
\[
\operatorname{sdefect}(J,3)
=
\sum_{1\leq i<j<k<\ell\leq t}
n_i n_j n_k n_\ell
+
\sum_{\substack{1\leq i<j\leq t\\ k\neq i,j}}
\binom{n_i+1}{2}\binom{n_j+1}{2}n_k.
\]
\end{corollary}

\begin{proof}
For \(s=2\), the partition classification of \(\mathcal D_s(K_t)\) gives the
single partition pattern $(1,1,1)$. For \(s=3\), it gives the two partition patterns $(1,1,1,1)$ and $(2,2,1)$.
Substituting these patterns into Corollary~\ref{cor:complete-multipartite-general} gives the stated formulas.
\end{proof}

\subsection{Complete split graphs}

Let \(S_{c,d}\) denote the complete split graph with a clique of size \(c\) and
an independent set of size \(d\), where every clique vertex is adjacent to every
independent vertex. Then
\[
S_{c,d}=K_{\underbrace{1,\ldots,1}_{c\text{ times}},d}.
\]
Hence \(S_{c,d}\) is a special blow-up of a complete graph.

\begin{corollary}\label{cor:complete-split-s2-s3}
Let $J=I(S_{c,d})$. Then
\[
\operatorname{sdefect}(J,2)
=
\binom{c}{3}+\binom{c}{2}d,
\]
and
\[
\operatorname{sdefect}(J,3)
=
\binom{c}{4}
+
\binom{c}{3}d
+
\binom{c}{2}(c-2+d)
+
c(c-1)\binom{d+1}{2}.
\]
\end{corollary}

\begin{proof}
Substitute $(n_1,\ldots,n_{c+1})=(1,\ldots,1,d)$ into Corollary~\ref{cor:complete-multipartite-s2-s3}.
\end{proof}

\subsection{Blow-ups of odd cycles}

The transfer formula also applies to non-complete base graphs. Let \(C_{2q+1}\) be the odd cycle on vertices $v_1,\ldots,v_{2q+1}$. It is known that
\[
I(C_{2q+1})^{(q+1)}
=
I(C_{2q+1})^{q+1}
+
(x_1x_2\cdots x_{2q+1});
\]
see, for example, the symbolic-power decompositions for odd cycles and
unicyclic graphs in \cite{FRA, GU, MAN}. Hence $\mathcal D_{q+1}(C_{2q+1})=\{(1,\ldots,1)\}$.
Applying Theorem~\ref{thm:blowup-transfer-symbolic-defect}, we obtain the
following formula for all blow-ups of odd cycles.

\begin{corollary}\label{cor:odd-cycle-blowup-first-defect}
Let $\mathbf n=(n_1,\ldots,n_{2q+1})$ and let $J=I(C_{2q+1}^{\mathbf n})$. Then
\[
\operatorname{sdefect}(J,q+1)
=
\prod_{i=1}^{2q+1}n_i.
\]
\end{corollary}

\begin{proof}
Since $\mathcal D_{q+1}(C_{2q+1})=\{(1,\ldots,1)\}$, Theorem~\ref{thm:blowup-transfer-symbolic-defect} gives
\[
\operatorname{sdefect}(J,q+1)
=
\prod_{i=1}^{2q+1}
\binom{1+n_i-1}{n_i-1}
=
\prod_{i=1}^{2q+1}n_i.
\]
\end{proof}

\begin{remark}
The complete multipartite and complete split formulas come from the base graph
\(K_t\), whereas Corollary~\ref{cor:odd-cycle-blowup-first-defect} comes from
the non-complete base graph \(C_{2q+1}\). This illustrates that the transfer
formula is not restricted to complete base graphs.
\end{remark}

\section{Symbolic defect antichains under graph joins}
\label{sec:joins}

In this section, we briefly record how the defect-region language behaves under
graph joins. This operation is included as a complementary construction to
graph blow-ups: while blow-ups replace vertices by independent sets, joins add
all possible edges between two graphs. The results below are stated in terms of
symbolic and ordinary exponent regions, and they provide a finite criterion for
computing defect antichains of joins.

Let \(G_1\) and \(G_2\) be finite simple graphs on disjoint vertex sets $V(G_1)=\{x_1,\ldots,x_r\}$, $V(G_2)=\{y_1,\ldots,y_m\}$.
The join of \(G_1\) and \(G_2\), denoted by $G_1*G_2$, is the graph with vertex set
\[
V(G_1*G_2)=V(G_1)\sqcup V(G_2)
\]
and edge set
\[
E(G_1*G_2)
=
E(G_1)\sqcup E(G_2)
\sqcup
\bigl\{\{x_i,y_j\}:1\leq i\leq r,\ 1\leq j\leq m\bigr\}.
\]
For vectors
\[
\mathbf a=(a_1,\ldots,a_r)\in\mathbb N^r,
\qquad
\mathbf b=(b_1,\ldots,b_m)\in\mathbb N^m,
\]
we write
\[
|\mathbf a|=a_1+\cdots+a_r,
\qquad
|\mathbf b|=b_1+\cdots+b_m.
\]

\subsection{Symbolic regions of joins}

We first describe the symbolic exponent region of a join.

\begin{lemma}\label{lem:min-covers-join}
Let $G=G_1*G_2$. Then the minimal vertex covers of \(G\) are precisely the sets $C_1\cup V(G_2)$, $C_1\in\mathcal C(G_1)$, and $V(G_1)\cup C_2$, $C_2\in\mathcal C(G_2)$, where \(\mathcal C(G_i)\) denotes the set of minimal vertex covers of \(G_i\).
\end{lemma}

\begin{proof}
Let \(W\) be a vertex cover of \(G_1*G_2\). Since every vertex of \(G_1\) is adjacent to every vertex of \(G_2\), the complement of \(W\) cannot contain vertices from both \(V(G_1)\) and \(V(G_2)\). Hence, either $V(G_1)\subseteq W$ or $V(G_2)\subseteq W$.
If \(V(G_2)\subseteq W\), then \(W\cap V(G_1)\) must be a vertex cover of \(G_1\), and minimality of \(W\) forces \(W\cap V(G_1)\) to be a minimal vertex cover of \(G_1\). Thus $W=C_1\cup V(G_2)$ for some \(C_1\in\mathcal C(G_1)\). The case \(V(G_1)\subseteq W\) is analogous. The converse is immediate from the definition of the join and the minimality of \(C_1\) or \(C_2\).
\end{proof}

For convenience, if \(q\leq 0\), we set $\mathcal P_q(G_i)=\mathbb N^{|V(G_i)|}$.

\begin{proposition}\label{prop:symbolic-region-join}
Let $G=G_1*G_2$. Then, for every \(s\geq 1\),
\[
\mathcal P_s(G)
=
\left\{
(\mathbf a,\mathbf b)\in\mathbb N^r\times\mathbb N^m:
\mathbf a\in\mathcal P_{s-|\mathbf b|}(G_1)
\text{ and }
\mathbf b\in\mathcal P_{s-|\mathbf a|}(G_2)
\right\}.
\]
Equivalently, \((\mathbf a,\mathbf b)\in\mathcal P_s(G)\) if and only if
\[
\sum_{x_i\in C_1}a_i+|\mathbf b|\geq s
\]
for every \(C_1\in\mathcal C(G_1)\), and
\[
|\mathbf a|+\sum_{y_j\in C_2}b_j\geq s
\]
for every \(C_2\in\mathcal C(G_2)\).
\end{proposition}

\begin{proof}
By Lemma~\ref{lem:min-covers-join}, the minimal vertex covers of \(G\) are
exactly $C_1\cup V(G_2)$ and $V(G_1)\cup C_2$.
Therefore, the symbolic-cover inequalities for \((\mathbf a,\mathbf b)\) are
precisely $\sum_{x_i\in C_1}a_i+|\mathbf b|\geq s$ for all \(C_1\in\mathcal C(G_1)\), and $|\mathbf a|+\sum_{y_j\in C_2}b_j\geq s$
for all \(C_2\in\mathcal C(G_2)\). These conditions are equivalent to $\mathbf a\in\mathcal P_{s-|\mathbf b|}(G_1)$
and $\mathbf b\in\mathcal P_{s-|\mathbf a|}(G_2)$, with the convention that \(\mathcal P_q(G_i)=\mathbb N^{|V(G_i)|}\) for \(q\leq 0\).
\end{proof}

\subsection{Ordinary regions of joins}

We now describe the ordinary-power exponent region. For \(p\geq 0\), define
\[
\Gamma_p(G_i)
=
\left\{
A_{G_i}\boldsymbol\lambda:
\boldsymbol\lambda\in\mathbb N^{E(G_i)}
\text{ and }
\sum_{e\in E(G_i)}\lambda_e=p
\right\},
\]
where \(A_{G_i}\) is the vertex-edge incidence matrix of \(G_i\). Thus \(\Gamma_p(G_i)\) records the exact exponent vectors of products of \(p\) edge monomials of \(G_i\). We use the convention $\Gamma_0(G_i)=\{\mathbf 0\}$.

\begin{proposition}\label{prop:ordinary-region-join}
Let $G=G_1*G_2$. For \(s\geq 1\), a vector $(\mathbf a,\mathbf b)\in\mathbb N^r\times\mathbb N^m$ belongs to \(\mathcal O_s(G)\) if and only if there exist nonnegative integers \(p,q,h\) with $p+q+h=s$, and vectors
\[
\mathbf c\in\Gamma_p(G_1),
\qquad
\mathbf d\in\Gamma_q(G_2),
\]
such that
\[
\mathbf c\leq\mathbf a,
\qquad
\mathbf d\leq\mathbf b,
\]
and
\[
|\mathbf a-\mathbf c|\geq h,
\qquad
|\mathbf b-\mathbf d|\geq h.
\]
\end{proposition}

\begin{proof}
Suppose first that $(\mathbf a,\mathbf b)\in\mathcal O_s(G)$. Then \((\mathbf a,\mathbf b)\) dominates the exponent vector of a product of \(s\) edge monomials of \(G\). Among these \(s\) edges, let \(p\) be the number
of edges inside \(G_1\), let \(q\) be the number of edges inside \(G_2\), and
let \(h\) be the number of cross edges between \(G_1\) and \(G_2\). Then $p+q+h=s$. The internal edges determine vectors $\mathbf c\in\Gamma_p(G_1)$, $\mathbf d\in\Gamma_q(G_2)$, with $\mathbf c\leq\mathbf a$, $\mathbf d\leq\mathbf b$. The \(h\) cross edges use exactly \(h\) vertex occurrences from the \(G_1\)-side
and \(h\) vertex occurrences from the \(G_2\)-side, so $|\mathbf a-\mathbf c|\geq h$ and $|\mathbf b-\mathbf d|\geq h$.
Conversely, suppose that such \(p,q,h,\mathbf c,\mathbf d\) exist. Choose a
product of \(p\) edge monomials of \(G_1\) with exponent vector \(\mathbf c\),
and a product of \(q\) edge monomials of \(G_2\) with exponent vector
\(\mathbf d\). Since $|\mathbf a-\mathbf c|\geq h$ and $|\mathbf b-\mathbf d|\geq h$,
we can choose \(h\) additional vertex occurrences from the residual degree on each side. Pairing these occurrences gives \(h\) cross edges, because every vertex of \(G_1\) is adjacent to every vertex of \(G_2\). Thus
\((\mathbf a,\mathbf b)\) dominates a product of \(s\) edge monomials of \(G\), and hence $(\mathbf a,\mathbf b)\in\mathcal O_s(G)$.
\end{proof}

\subsection{The defect antichain of a join}

Combining the preceding two descriptions gives a finite criterion for the
defect antichain of a join.

\begin{theorem}[Defect antichain criterion for joins]
\label{thm:defect-antichain-join}
Let $G=G_1*G_2$. For every \(s\geq 1\),
$\mathcal D_s(G)
=
\min_{\leq}
\bigl(\mathcal P_s(G)\setminus\mathcal O_s(G)\bigr)$,
where \(\mathcal P_s(G)\) is described in Proposition~\ref{prop:symbolic-region-join} and \(\mathcal O_s(G)\) is
described in Proposition~\ref{prop:ordinary-region-join}. Equivalently,
$(\mathbf a,\mathbf b)\in\mathcal D_s(G)$ if and only if the following three conditions hold:

\begin{enumerate}
\item
\[
\mathbf a\in\mathcal P_{s-|\mathbf b|}(G_1)
\quad\text{and}\quad
\mathbf b\in\mathcal P_{s-|\mathbf a|}(G_2);
\]

\item there do not exist \(p,q,h\in\mathbb N\) with \(p+q+h=s\), and vectors
\[
\mathbf c\in\Gamma_p(G_1),
\qquad
\mathbf d\in\Gamma_q(G_2),
\]
such that
\[
\mathbf c\leq\mathbf a,\qquad
\mathbf d\leq\mathbf b,
\]
and
\[
|\mathbf a-\mathbf c|\geq h,\qquad
|\mathbf b-\mathbf d|\geq h;
\]

\item no proper componentwise smaller vector
\[(\mathbf a',\mathbf b')<(\mathbf a,\mathbf b)
\]
satisfies conditions \((1)\) and \((2)\).
\end{enumerate}
\end{theorem}

\begin{proof}
By Definition~\ref{def:defect-antichain}, the defect antichain is the set of minimal elements of $\mathcal P_s(G)\setminus\mathcal O_s(G)$.
Condition \((1)\) is exactly Proposition~\ref{prop:symbolic-region-join}.
Condition \((2)\) is the negation of Proposition~\ref{prop:ordinary-region-join}.
Condition \((3)\) is minimality with respect to the componentwise order. Hence
the stated criterion follows.
\end{proof}

\begin{remark}
Theorem~\ref{thm:defect-antichain-join} is deliberately stated as a region-level
criterion rather than as a closed formula. For arbitrary joins, edge choices
inside \(G_1\), inside \(G_2\), and across the join interact nontrivially. The
criterion is nevertheless useful because it reduces the computation of
\(\mathcal D_s(G_1*G_2)\) to the symbolic regions and edge-incidence semigroups
of the two factors.
\end{remark}

\subsection{Pure-side defect patterns}

The following consequence shows that some defect patterns of the factors persist
inside the join.

\begin{proposition}\label{prop:pure-side-defects-join}
Let $G=G_1*G_2$. Fix \(s\geq 1\). If $\mathbf a\in\mathcal D_s(G_1)$ and $|\mathbf a|\geq s$, then
$(\mathbf a,\mathbf 0)\in\mathcal D_s(G)$. Similarly, if $\mathbf b\in\mathcal D_s(G_2)$ and $|\mathbf b|\geq s$,
then $(\mathbf 0,\mathbf b)\in\mathcal D_s(G)$.
\end{proposition}

\begin{proof}
We prove the first assertion; the second is analogous. Since $\mathbf a\in\mathcal D_s(G_1)$, we have
$\mathbf a\in\mathcal P_s(G_1)\setminus\mathcal O_s(G_1)$. Moreover, the assumption \(|\mathbf a|\geq s\) implies $\mathbf 0\in\mathcal P_{s-|\mathbf a|}(G_2)$ by the convention for nonpositive symbolic indices. Hence Proposition~\ref{prop:symbolic-region-join} gives $(\mathbf a,\mathbf 0)\in\mathcal P_s(G)$. If $(\mathbf a,\mathbf 0)\in\mathcal O_s(G)$,
then Proposition~\ref{prop:ordinary-region-join} forces \(q=h=0\) and \(p=s\),
because there is no available degree on the \(G_2\)-side. Thus $\mathbf a\in\mathcal O_s(G_1)$, contradicting
$\mathbf a\in\mathcal D_s(G_1)$. Therefore $(\mathbf a,\mathbf 0)\in\mathcal P_s(G)\setminus\mathcal O_s(G)$. Finally, suppose $(\mathbf a',\mathbf b')\leq(\mathbf a,\mathbf 0)$ and $(\mathbf a',\mathbf b')\in\mathcal P_s(G)\setminus\mathcal O_s(G)$.
Then \(\mathbf b'=\mathbf 0\). By the previous argument, $\mathbf a'\in\mathcal P_s(G_1)\setminus\mathcal O_s(G_1)$. Since \(\mathbf a\) is minimal in this region, \(\mathbf a'=\mathbf a\). Hence $(\mathbf a,\mathbf 0)$ is minimal, and so $(\mathbf a,\mathbf 0)\in\mathcal D_s(G)$.
\end{proof}

\begin{remark}
This proposition is included only to illustrate that the antichain language is
compatible with graph joins. The main results of the paper remain the blow-up
transfer theorem, the multigraded defect series, and the partition
classification for complete graphs.
\end{remark}

\section{Asymptotic and quasi-polynomial consequences}
\label{sec:asymptotic}

In this section we record several asymptotic consequences of the blow-up
transfer formula. The main point is that, for fixed symbolic degree \(s\), the
symbolic defect of a graph blow-up is polynomial in the blow-up parameters. We
also explain how the defect-antichain framework is compatible with the known
study of eventual quasi-polynomial behavior of symbolic defect functions
\cite{DRA, MAN, OLT}.

\subsection{Polynomiality in the blow-up parameters}

Let \(G\) be a finite simple graph with vertex set $V(G)=\{v_1,\ldots,v_r\}$. For $\mathbf n=(n_1,\ldots,n_r)\in\mathbb N^r$, let \(G^{\mathbf n}\) be the \(\mathbf n\)-blow-up of \(G\).

\begin{theorem}[Polynomiality in the blow-up parameters]
\label{thm:multivariable-polynomiality}
Fix \(s\geq 1\). Then
\[
(n_1,\ldots,n_r)
\longmapsto
\operatorname{sdefect}(I(G^{\mathbf n}),s)
\]
is a polynomial function in \(n_1,\ldots,n_r\). More precisely,
\[
\operatorname{sdefect}(I(G^{\mathbf n}),s)
=
\sum_{\mathbf a\in\mathcal D_s(G)}
\prod_{i=1}^{r}
\binom{a_i+n_i-1}{a_i}.
\]
Its total degree is at most $M_s(G)=\max\{|\mathbf a|:\mathbf a\in\mathcal D_s(G)\}$, where
$|\mathbf a|=a_1+\cdots+a_r$.
\end{theorem}

\begin{proof}
By Theorem~\ref{thm:blowup-transfer-symbolic-defect},
\[
\operatorname{sdefect}(I(G^{\mathbf n}),s)
=
\sum_{\mathbf a\in\mathcal D_s(G)}
\prod_{i=1}^{r}
\binom{a_i+n_i-1}{n_i-1}.
\]
Since
\[
\binom{a_i+n_i-1}{n_i-1}
=
\binom{a_i+n_i-1}{a_i},
\]
and \(a_i\) is fixed for each \(\mathbf a\in\mathcal D_s(G)\), each factor is a
polynomial in \(n_i\) of degree \(a_i\). Therefore, each product has total degree $|\mathbf a|=a_1+\cdots+a_r$.
Since \(\mathcal D_s(G)\) is finite, the sum is a polynomial of total degree at
most \(M_s(G)\).
\end{proof}

\begin{corollary}[Uniform blow-ups]
\label{cor:uniform-blowup-polynomiality}
Let $G^{(q)}=G^{(q,\ldots,q)}$ be the uniform \(q\)-fold blow-up of \(G\). For fixed \(s\geq 1\), $q\longmapsto \operatorname{sdefect}(I(G^{(q)}),s)$ is a polynomial function of \(q\), given by
\[
\operatorname{sdefect}(I(G^{(q)}),s)
=
\sum_{\mathbf a\in\mathcal D_s(G)}
\prod_{i=1}^{r}
\binom{a_i+q-1}{a_i}.
\]
Its degree is at most \(M_s(G)\).
\end{corollary}

\begin{proof}
This is Theorem~\ref{thm:multivariable-polynomiality} with $n_1=\cdots=n_r=q$.
\end{proof}

\subsection{Generating functions in the blow-up parameters}

The polynomiality above can be equivalently expressed through rational
generating functions. For fixed \(G\) and \(s\), define
\[
\mathcal B_{G,s}(z_1,\ldots,z_r)
=
\sum_{n_1,\ldots,n_r\geq 1}
\operatorname{sdefect}(I(G^{\mathbf n}),s)
z_1^{n_1}\cdots z_r^{n_r}.
\]

\begin{theorem}[Blow-up defect generating function]
\label{thm:blowup-defect-generating-function}
For every \(s\geq 1\),
\[
\mathcal B_{G,s}(z_1,\ldots,z_r)
=
\sum_{\mathbf a\in\mathcal D_s(G)}
\prod_{i=1}^{r}
\frac{z_i}{(1-z_i)^{a_i+1}}.
\]
In particular, \(\mathcal B_{G,s}(z_1,\ldots,z_r)\) is a rational function.
\end{theorem}

\begin{proof}
Using Theorem~\ref{thm:multivariable-polynomiality}, we have
\[
\operatorname{sdefect}(I(G^{\mathbf n}),s)
=
\sum_{\mathbf a\in\mathcal D_s(G)}
\prod_{i=1}^{r}
\binom{a_i+n_i-1}{a_i}.
\]
Hence
\[
\mathcal B_{G,s}(z_1,\ldots,z_r)
=
\sum_{\mathbf a\in\mathcal D_s(G)}
\prod_{i=1}^{r}
\left(
\sum_{n_i\geq 1}
\binom{a_i+n_i-1}{a_i}z_i^{n_i}
\right).
\]
The standard identity
\[
\sum_{n_i\geq 1}
\binom{a_i+n_i-1}{a_i}z_i^{n_i}
=
\frac{z_i}{(1-z_i)^{a_i+1}}
\]
gives the stated formula.
\end{proof}

For uniform blow-ups, define $\mathcal U_{G,s}(z)=\sum_{q\geq 1}\operatorname{sdefect}(I(G^{(q)}),s)z^q$.

\begin{corollary}[Uniform blow-up generating function]
\label{cor:uniform-blowup-generating-function}
For fixed \(G\) and \(s\), the series \(\mathcal U_{G,s}(z)\) is rational. Moreover, if \(\mathcal D_s(G)\neq\emptyset\), then its denominator divides $(1-z)^{M_s(G)+1}$.
\end{corollary}

\begin{proof}
By Corollary~\ref{cor:uniform-blowup-polynomiality}, $q\longmapsto \operatorname{sdefect}(I(G^{(q)}),s)$ is a polynomial function of degree at most \(M_s(G)\). The generating function of a polynomial function of degree at most \(M_s(G)\) has a denominator dividing $(1-z)^{M_s(G)+1}$.
\end{proof}

\subsection{Eventual quasi-polynomial behavior}

We now discuss the dependence on the symbolic power \(s\). Eventual quasi-polynomial behavior of symbolic defect functions has been studied in several settings, including ideals with Noetherian symbolic Rees algebra, symbolic defects of unicyclic edge ideals, and symbolic-polyhedral approaches to monomial ideals \cite{DRA, MAN, OLT}. The result below is a transfer principle: it does not assert eventual quasi-polynomiality for all graphs, but shows that any such behavior at the level of base defect
antichains is inherited by fixed blow-ups.

For a polynomial $w(\mathbf x)\in\mathbb Q[x_1,\ldots,x_r]$, define the weighted defect-antichain sum
\[
\Phi_{G,w}(s)
=
\sum_{\mathbf a\in\mathcal D_s(G)}w(\mathbf a).
\]

\begin{theorem}[Quasi-polynomial transfer]
\label{thm:quasi-polynomial-transfer}
Let \(G\) be a finite simple graph and fix $\mathbf n=(n_1,\ldots,n_r)\in\mathbb N^r$. Assume that for every polynomial $w(\mathbf x)\in\mathbb Q[x_1,\ldots,x_r]$, the function $s\longmapsto \Phi_{G,w}(s)$ is eventually quasi-polynomial. Then $s\longmapsto
\operatorname{sdefect}(I(G^{\mathbf n}),s)$
is eventually quasi-polynomial.
\end{theorem}

\begin{proof}
For the fixed vector \(\mathbf n\), define
\[
w_{\mathbf n}(\mathbf x)
=
\prod_{i=1}^{r}
\binom{x_i+n_i-1}{n_i-1}.
\]
Since each \(n_i\) is fixed, \(w_{\mathbf n}(\mathbf x)\) is a polynomial in
\(\mathbf x\). By the blow-up transfer formula,
\[
\operatorname{sdefect}(I(G^{\mathbf n}),s)
=
\sum_{\mathbf a\in\mathcal D_s(G)}
w_{\mathbf n}(\mathbf a)
=
\Phi_{G,w_{\mathbf n}}(s).
\]
By hypothesis, this function is eventually quasi-polynomial in \(s\).
\end{proof}

\begin{remark}
Theorem~\ref{thm:quasi-polynomial-transfer} is intentionally conditional. Its role is to separate the behavior of the base antichains $\mathcal D_s(G)$ from the effect of the blow-up parameters. Once weighted counts of the base
defect antichains are eventually quasi-polynomial, the symbolic defect
functions of all fixed blow-ups inherit the same type of behavior.
\end{remark}
\subsection{Eventual quasi-polynomiality for complete multipartite graphs}

We now prove eventual quasi-polynomiality for the symbolic defect functions of complete multipartite graphs. This proves eventual quasi-polynomiality for an important class of graph blow-ups.

We use the following standard fact from the theory of rational generating functions. If \(P\subseteq \mathbb R^d\) is a rational polyhedron and \(w\) is a polynomial function on \(\mathbb Z^d\), then the weighted generating function
\[
\sum_{\mathbf u\in P\cap\mathbb Z^d} w(\mathbf u)z^{\ell(\mathbf u)}
\]
is rational, whenever \(\ell\) is an integral linear form bounded below on \(P\cap\mathbb Z^d\). Consequently, its coefficient function is eventually quasi-polynomial \cite{STA, BEC}.

\begin{theorem}\label{thm7.7}
Let $J=I(K_{n_1,\ldots,n_t})$ be the edge ideal of a complete multipartite graph. Then the ordinary generating function
\[
F_J(z)=\sum_{s\geq 1}\operatorname{sdefect}(J,s)z^s
\]
is a rational function. Consequently, the function $s\longmapsto \operatorname{sdefect}(J,s)$ is eventually quasi-polynomial.
\end{theorem}

\begin{proof}
Since $K_{n_1,\ldots,n_t}=K_t^{(n_1,\ldots,n_t)}$, the blow-up transfer formula gives
\[
\operatorname{sdefect}(J,s)
=
\sum_{\mathbf a\in\mathcal D_s(K_t)}
\prod_{i=1}^t
\binom{a_i+n_i-1}{n_i-1}.
\]
By the partition classification of \(\mathcal D_s(K_t)\), a vector \(\mathbf a=(a_1,\ldots,a_t)\) belongs to \(\mathcal D_s(K_t)\) if and only if there exists an integer \(p\) such that $|\mathbf a|=s+p$, $\max(\mathbf a)=p$, the maximum value \(p\) occurs at least twice and \(1\leq p\leq s-1\).
For a subset \(M\subseteq \{1,\ldots,t\}\) with \(|M|\geq 2\), consider those
vectors for which the set of coordinates attaining the maximum value is exactly
\(M\). Then $a_i=p$ for $i\in M$, and $0\leq a_j\leq p-1$ for $j\notin M$. For such a vector,
\[
s=|\mathbf a|-p=(|M|-1)p+\sum_{j\notin M}a_j.
\]
The condition \(p\leq s-1\) is equivalent to $(|M|-2)p+\sum_{j\notin M}a_j\geq 1$. Therefore,
\[
F_J(z)=
\sum_{\substack{M\subseteq \{1,\ldots,t\}\\ |M|\geq 2}}
\sum_{\substack{p\geq 1,\;0\leq a_j\leq p-1\\
j\notin M,\; (|M|-2)p+\sum_{j\notin M}a_j\geq 1}}
\left(
\prod_{i\in M}\binom{p+n_i-1}{n_i-1}
\prod_{j\notin M}\binom{a_j+n_j-1}{n_j-1}
\right)
z^{(|M|-1)p+\sum_{j\notin M}a_j}.
\]
This expression partitions the defect vectors according to the exact set \(M\)
of coordinates attaining the maximum, so no vector is counted more than once.

For each fixed \(M\), the summation is over the integer points of a rational
polyhedral set defined by linear inequalities. The weight
\[
\prod_{i\in M}\binom{p+n_i-1}{n_i-1}
\prod_{j\notin M}\binom{a_j+n_j-1}{n_j-1}
\]
is a polynomial in the variables \(p\) and \(a_j\). Hence, by the standard
rationality theorem for weighted integer-point generating functions over
rational polyhedral sets, each summand indexed by \(M\) is a rational function
of \(z\). Since there are only finitely many subsets \(M\), \(F_J(z)\) is rational.

A rational generating function of this type eventually has quasi-polynomial
coefficients. Hence $s\longmapsto \operatorname{sdefect}(J,s)$ is eventually quasi-polynomial.
\end{proof}

\begin{corollary}
For every \(t\geq 2\), the function $s\longmapsto \operatorname{sdefect}(I(K_t),s)$ is eventually quasi-polynomial.
\end{corollary}

\begin{proof}
This is the special case \(n_1=\cdots=n_t=1\) of the preceding theorem.
\end{proof}

\begin{corollary}
Let \(S_{c,d}\) be the complete split graph with a clique of size \(c\) and an independent set of size \(d\). Then $s\longmapsto \operatorname{sdefect}(I(S_{c,d}),s)$ is eventually quasi-polynomial.
\end{corollary}

\begin{proof}
Since \(S_{c,d}=K_{1,\ldots,1,d}\), it is a complete multipartite graph. The result follows from the theorem.
\end{proof}

\subsection{Relation with the multigraded defect series}

The preceding results can be viewed as numerical specializations of the multigraded defect series. Indeed, Theorem~\ref{thm:multigraded-defect-series}
shows that
\[
\mathcal M_{I(G^{\mathbf n}),s}(\mathbf z)
=
\sum_{\mathbf a\in\mathcal D_s(G)}
\prod_{i=1}^{r}
h_{a_i}(z_{i1},\ldots,z_{in_i}).
\]
Setting all auxiliary variables equal to \(1\) gives the numerical symbolic
defect: $\operatorname{sdefect}(I(G^{\mathbf n}),s)=\mathcal M_{I(G^{\mathbf n}),s}(1,\ldots,1)$.
Thus, the polynomiality and rationality results above are consequences of the same underlying antichain data. This reinforces the central point of the paper: the symbolic defect antichain is a finer invariant than the numerical symbolic defect function.

\section{Examples and computations}
\label{sec:examples}

In this section we illustrate the preceding results with explicit computations.
The purpose of these examples is not to introduce additional graph families, but
to show how the symbolic defect antichain of a base graph produces concrete
formulas for its blow-ups.
All computations were performed using Macaulay2 \cite{MAC}.

\subsection{Example: A blow-up of \(K_3\)}

Let $G=K_{2,3,4}$. Since $K_{2,3,4}=K_3^{(2,3,4)}$, the symbolic defects of \(G\) are computed from the defect antichains of the
base graph \(K_3\).

For \(s=2\), Corollary~\ref{cor:complete-multipartite-s2-s3} gives $\operatorname{sdefect}(I(K_{2,3,4}),2)
=2\cdot 3\cdot 4=24$.

For \(s=3\), the same corollary gives
\[
\operatorname{sdefect}(I(K_{2,3,4}),3)
=
\sum_{\substack{1\leq i<j\leq 3\\ k\neq i,j}}
\binom{n_i+1}{2}\binom{n_j+1}{2}n_k,
\]
where
\[
(n_1,n_2,n_3)=(2,3,4).
\]
Thus
\[
\operatorname{sdefect}(I(K_{2,3,4}),3)
=
\binom{3}{2}\binom{4}{2}\cdot 4
+
\binom{3}{2}\binom{5}{2}\cdot 3
+
\binom{4}{2}\binom{5}{2}\cdot 2.
\]
Hence $\operatorname{sdefect}(I(K_{2,3,4}),2)=24$ and $\operatorname{sdefect}(I(K_{2,3,4}),3)=282$.

\subsection{Example: A complete-graph blow-up giving complete split graphs}
\label{subsec:example-complete-split}

Let \(S_{c,d}\) denote the complete split graph with a clique of size \(c\) and
an independent set of size \(d\), where every clique vertex is adjacent to
every independent vertex. Then
\[S_{c,d}=K_{\underbrace{1,\ldots,1}_{c\text{ times}},d}.
\]
Thus \(S_{c,d}\) is a special blow-up of a complete graph.

For \(s=2\), Corollary~\ref{cor:complete-split-s2-s3} gives
\[\operatorname{sdefect}(I(S_{c,d}),2)
=
\binom{c}{3}+\binom{c}{2}d.
\]
For \(s=3\), the same corollary gives
\[
\operatorname{sdefect}(I(S_{c,d}),3)
=
\binom{c}{4}
+
\binom{c}{3}d
+
\binom{c}{2}(c-2+d)
+
c(c-1)\binom{d+1}{2}.
\]

For example, take \(c=3\) and \(d=4\). Then
\[
\operatorname{sdefect}(I(S_{3,4}),2)
=
\binom{3}{3}+\binom{3}{2}\cdot 4
=
1+12
=
13.
\]
Moreover,
\[
\operatorname{sdefect}(I(S_{3,4}),3)
=
\binom{3}{4}
+
\binom{3}{3}\cdot 4
+
\binom{3}{2}(3-2+4)
+
3\cdot 2\binom{5}{2}.
\]
Hence $\operatorname{sdefect}(I(S_{3,4}),3)=0+4+15+60=79$.

This example illustrates that complete split graphs are not treated as a
separate family. They arise naturally as special blow-ups of complete graphs,
and their symbolic defect formulas are obtained by specializing the general
blow-up transfer formula.

\subsection{Example: blow-ups of the odd cycle \(C_5\)}

Let $G=C_5$ with vertex set $V(C_5)=\{v_1,v_2,v_3,v_4,v_5\}$. For the odd cycle \(C_5\), the first symbolic defect occurs at \(s=3\), and $\mathcal D_3(C_5)=\{(1,1,1,1,1)\}$. Therefore, by Corollary~\ref{cor:odd-cycle-blowup-first-defect},
\[
\operatorname{sdefect}
\left(
I\left(C_5^{(n_1,n_2,n_3,n_4,n_5)}\right),3
\right)
=
n_1n_2n_3n_4n_5.
\]

For instance, if $(n_1,n_2,n_3,n_4,n_5)=(2,1,3,1,2)$, then
\[
\operatorname{sdefect}
\left(
I\left(C_5^{(2,1,3,1,2)}\right),3
\right)
=
2\cdot 1\cdot 3\cdot 1\cdot 2
=
12.
\]

This example shows that the blow-up transfer formula is not restricted to
complete base graphs. It also applies to non-complete non-bipartite base
graphs.

\subsection{Example: Uniform blow-ups}

We now illustrate the polynomiality result for uniform blow-ups. Let $G^{(q)}=G^{(q,\ldots,q)}$. First take $G=K_3$. Then $K_3^{(q)}=K_{q,q,q}$. For \(s=2\), Corollary~\ref{cor:complete-multipartite-s2-s3} gives
$\operatorname{sdefect}(I(K_{q,q,q}),2)
=
q^3$.
For \(s=3\), the same corollary gives
\[
\operatorname{sdefect}(I(K_{q,q,q}),3)
=
3q\binom{q+1}{2}^2.
\]
Thus, the first two symbolic defects are polynomial functions of the uniform
blow-up parameter \(q\), in agreement with
Corollary~\ref{cor:uniform-blowup-polynomiality}.
For example, when \(q=2\), $\operatorname{sdefect}(I(K_{2,2,2}),2)=8$, and $\operatorname{sdefect}(I(K_{2,2,2}),3)=3\cdot 2\binom{3}{2}^2=54$.

\subsection{Summary of computations}

The preceding computations are summarized in
Table~\ref{tab:example-computations}. They illustrate how different base graphs
lead to different symbolic defect patterns, while all computations are obtained
from the same transfer principle.

\begin{table}[ht]
\centering
\caption{Sample symbolic defect computations from the transfer formula.}
\label{tab:example-computations}
\begin{tabular}{llll}
\toprule
Graph \(G\) & Base graph & \(s\) & \(\operatorname{sdefect}(I(G),s)\) \\
\midrule
\(K_{2,3,4}\) & \(K_3\) & \(2\) & \(24\) \\
\(K_{2,3,4}\) & \(K_3\) & \(3\) & \(282\) \\
\(S_{3,4}\) & \(K_4\) blow-up & \(2\) & \(13\) \\
\(S_{3,4}\) & \(K_4\) blow-up & \(3\) & \(79\) \\
\(C_5^{(2,1,3,1,2)}\) & \(C_5\) & \(3\) & \(12\) \\
\(K_{2,2,2}\) & \(K_3\) & \(2\) & \(8\) \\
\(K_{2,2,2}\) & \(K_3\) & \(3\) & \(54\) \\
\bottomrule
\end{tabular}
\end{table}

These examples demonstrate three features of the method. First, complete
multipartite and complete split graphs arise as applications of the blow-up
framework rather than as isolated families. Second, odd-cycle blow-ups show that
the method applies to non-complete base graphs. Third, uniform blow-ups exhibit
the polynomial behavior predicted by the asymptotic results.
\section{Conclusion}
\label{sec:conclusion}

We introduced symbolic defect antichains as a structural refinement of symbolic defect functions for edge ideals. For a graph \(G\), the antichain $\mathcal D_s(G)=\min_{\leq}\bigl(\mathcal P_s(G)\setminus\mathcal O_s(G)\bigr)$
records the minimal exponent vectors responsible for the quotient \(I(G)^{(s)}/I(G)^s\). Thus \(\mathcal D_s(G)\) retains more information than the numerical invariant \(\operatorname{sdefect}(I(G),s)\), while remaining a
finite obstruction set.
The main result shows that this obstruction data transfers through arbitrary graph blow-ups. For \(G^{\mathbf n}\), the symbolic defect of \(I(G^{\mathbf n})\) is a weighted sum over \(\mathcal D_s(G)\), where the
weights depend only on the blow-up sizes. We also refined this numerical formula to a multigraded defect-generator series, showing that the transfer theorem describes the full multigraded distribution of minimal generators of
\(I(G^{\mathbf n})^{(s)}/I(G^{\mathbf n})^s\).
We classified the antichains \(\mathcal D_s(K_t)\) for complete graphs in all symbolic degrees using partition types. Since complete multipartite graphs are blow-ups of complete graphs, this yields all-degree formulas for their symbolic defects and multigraded defect series. Complete split graphs and blow-ups of odd
cycles then arise as further applications of the same transfer principle.
The antichain viewpoint also gives polynomiality and rational generating-function consequences in the blow-up parameters. It suggests several directions for future work, including effective computation of \(\mathcal D_s(G)\), explicit classifications for additional graph classes, and connections with symbolic polyhedra and asymptotic symbolic defect functions.

Theorem~\ref{thm7.7} gives a positive answer to the quasi-polynomiality problem for complete multipartite graphs. For arbitrary finite simple graphs, the following question remains open.

\begin{question}
Let \(G\) be a finite simple graph. Is the function
\[
s\longmapsto \operatorname{sdefect}(I(G),s)
\]
eventually quasi-polynomial? More generally, can the symbolic defect antichains
\(\mathcal D_s(G)\) be described, for all sufficiently large \(s\), by finitely many affine-linear families?
\end{question}

\section*{Statements and Declarations}

\subsection*{Funding} 
The authors received no specific funding for this work.
\subsection*{Competing interests}
The authors declare that they have no competing interests.
\subsection*{Data availability}
No datasets were generated or analyzed during the current study.

\end{document}